\documentclass[12pt, a4paper]{amsart}
\usepackage[utf8]{inputenc}
\usepackage{amsfonts,amssymb,amsxtra,amsthm,amsmath,amscd,mathrsfs}
\usepackage[all]{xy}

\usepackage{tikz}

\usepackage[normalem]{ulem}
\usepackage{soul}
\usepackage{color}

\setstcolor{red}

\makeatletter
\@namedef{subjclassname@2010}{%
  \textup{2010} Mathematics Subject Classification}
\makeatother

\def\leq{\leqslant}
\def\geq{\geqslant}

\theoremstyle{plain}
\newtheorem{theorem}{Theorem}[section]
\newtheorem{proposition}{Proposition}[section]
\newtheorem{lemma}[proposition]{Lemma}

\theoremstyle{remark}

\numberwithin{equation}{section}

\addtocounter{footnote}{1}

\begin{document}

\vglue -5mm

\title[MDC on the product of a circle and  the Heisenberg nilmanifold]
{M{\"o}bius Disjointness Conjecture on the product of \\
a circle and the Heisenberg nilmanifold}
\author{Jing Ma \& Ronghui Wu}

\address{%
Jing Ma
\\
School of Mathematics
\\
Jilin University
\\
Changchun 130012
\\
P. R. China
}
\email{jma@jlu.edu.cn}

\address{
Ronghui Wu
\\
School of Mathematics
\\
Jilin University
\\
Changchun 130012
\\
P. R. China
}
\email{rhwu21@mails.jlu.edu.cn}

\date{\today}

\begin{abstract}
Let $\mathbb{T}$ be the unit circle and $\Gamma\backslash G$ the $3$-dimensional Heisenberg
nilmanifold. We prove that the M{\"o}bius function is linearly disjoint from a class of
distal skew products on $\mathbb{T}\times\Gamma\backslash G$.
These results generalize a recent work of Huang-Liu-Wang.
\end{abstract}


\maketitle

\section{{Introduction}}

Let $(X, T)$ be a flow, where $X$ is a compact metric space and $T: X\rightarrow X$ a continuous map.
Let $\mu(n)$ be the M{\"o}bius function, that is $\mu(n)=0$ if $n$ is not square-free,
and  $\mu(n)=(-1)^q$ if $n$ is a product of $q$ distinct primes.
$\mu$ is said to be linearly disjoint from $(X, T)$ if
\[
\lim_{N\rightarrow\infty}\frac 1N\sum_{n\leq N} \mu(n)f(T^{n}x)=0
\]
for any $f\in C(X)$ and any $x\in X$.
The M{\"o}bius Disjointness Conjecture of Sarnak \cite{Sarnak, Sarnak2} states that the function $\mu$ is linearly disjoint from every $(X, T)$ whose
entropy is $0$.
This conjecture has been established for many cases, and we refer to the survey paper \cite{Ferenczi} for recent progresses.
An incomplete list for works related to the present paper is:
Bourgain \cite{Bourgain2},  Bourgain-Sarnak-Ziegler \cite{Bourgain3}, Green-Tao \cite{GT}, Liu-Sarnak \cite{Liu, Liu2},
Wang \cite{Wang}, Peckner \cite{Peckner}, Huang-Wang-Ye \cite{HWY}, Litman-Wang \cite{Litman}, and Huang-Liu-Wang \cite{HLW}.

Distal flows are typical examples of zero-entropy flows, see Parry \cite{Parry}. A flow
$(X, T)$ with a compatible metric $d$ is called distal if
\[
\inf_{n\geq0}d(T^{n}x, T^{n}y)>0
\]
whenever $x\neq y. $  According to Furstenberg’s structure theorem of minimal distal flows \cite{Furstenberg2}, skew products are building blocks of distal flows.

An example of distal flow is the skew product $T$ on the $2$-torus $\mathbb{T}^2=(\mathbb{R}/\mathbb{Z})^2$ given by
\begin{equation}\label{jia}
T:(x, y)\mapsto(x+\alpha, y+h(x)),
\end{equation}
where $\alpha\in[0, 1)$ and $h: \mathbb{T}\rightarrow\mathbb{R}$ is a continuous function. For dynamical properties of this skew product, see for example Furstenberg \cite{Furstenberg}. The M{\"o}bius disjointness for the skew product \eqref{jia} was first studied by Liu and Sarnak in \cite{Liu, Liu2}. A result in \cite{Liu} states that, if $h$ is analytic with an additional assumption on its Fourier coefficients, then the M{\"o}bius Disjointness Conjecture is true for the skew product $(\mathbb{T}^2, T). $ This result holds for all $\alpha$, as is not common in the KAM theory. The aforementioned additional assumption was removed in Wang \cite{Wang}. It has been further generalized by Huang-Wang-Ye \cite{HWY} to the case that $h(x)$ is $C^\infty$-smooth.

Another example of distal flow is nilsystem. Let $G$ be a nilpotent Lie group
with a discrete cocomapct subgroup $\Gamma$. The group $G$ acts in a natural way on
the homogeneous space $\Gamma\backslash G. $ Fix $h\in G. $ Then the transformation $T$ given by
$T(\Gamma g)=\Gamma gh$ makes $(\Gamma\backslash G, T)$ a nilsystem.
The M{\"o}bius Disjointness Conjecture for these nilsystems was proved by Green and Tao in \cite{GT}.

Now let $G$ be the $3$-dimensional Heisenberg group with the cocompact discrete subgroup $\Gamma, $ namely
\[
G=
\tiny{
\begin {pmatrix}
1 & \mathbb{R} & \mathbb{R}\\
0 & 1 & \mathbb{R}\\
0 & 0 & 1
\end {pmatrix}, \qquad
\Gamma=\begin {pmatrix}
1 & \mathbb{Z} & \mathbb{Z}\\
0 & 1 & \mathbb{Z}\\
0 & 0 & 1
\end {pmatrix}
}.
\]
Then $\Gamma\backslash G$ is the $3$-dimensional Heisenberg nilmanifold. Let $\mathbb{T}$ be the unit circle.
Huang-Liu-Wang \cite{HLW} define a skew product $S$ on $\mathbb{T}\times\Gamma\backslash G$
\begin{equation}\label{S}
S:(t, \Gamma g)\mapsto
\left(
t +\alpha, \Gamma g
\footnotesize{
\begin{pmatrix}
1 & \varphi_2(t) & \psi(t) \\
0 & 1 & \varphi_1(t) \\
0 & 0 & 1
\end{pmatrix}
}
\right),
\end{equation}
where $\varphi_1$, $\varphi_2$ and $\psi$ are $C^{\infty}$-smooth periodic functions with period $1$.
They proved that the flow $(\mathbb{T}\times\Gamma\backslash G, S)$ is distal,
and in the case that $\varphi_1=\varphi_2=\phi$, the skew product
\[
S_0:(t, \Gamma g)\mapsto
\left(
t+\alpha, \Gamma g
\footnotesize{
\begin{pmatrix}
1 & \phi(t) & \psi(t) \\
0 & 1 & \phi(t) \\
0 & 0 & 1
\end{pmatrix}
}\right)
\]
is linearly disjoint with the M{\"o}bius function $\mu$.

The entropy of $(\mathbb{T}\times\Gamma\backslash G, S)$ is a zero, so  the M{\"o}bius Disjointness Conjecture is expected to hold on $(\mathbb{T}\times\Gamma\backslash G, S)$.
In this manuscript, we generalize the result of Huang-Liu-Wang \cite{HLW} by showing that the M{\"o}bius Disjointness Conjecture holds on $(\mathbb{T}\times\Gamma\backslash G,S)$ if $\alpha$ is rational.
Let the skew product $T$ on $\mathbb{T}\times\Gamma\backslash G$ be given by
\begin{equation}\label{T}
T:(t, \Gamma g)\mapsto
\left(
t+\alpha, \Gamma g
\footnotesize{
\begin{pmatrix}
1 & k\phi(t) & \psi(t) \\
0 & 1 & \phi(t) \\
0 & 0 & 1
\end{pmatrix}
}
\right),
\end{equation}
where  $\alpha\in[0, 1)$, $k\in\mathbb{R}$ and $\phi, \psi$ be $C^\infty$-smooth periodic functions from $\mathbb{R}$ to $\mathbb{R}$ with period $1$.
We also prove that the M{\"o}bius Disjointness Conjecture holds on $(\mathbb{T}\times\Gamma\backslash G,T)$.
In particular, the main results of this manuscript are as follows.

\begin{theorem}\label{thmrational}
Let $\mathbb{T}$ be the unit circle and $\Gamma\backslash G$ the $3$-dimensional Heisenberg nilmanifold.
Let $\alpha\in\mathbb{Q}\cap [0, 1)$,  $\varphi_1, \varphi_2, \psi$ be $C^\infty$-smooth periodic functions from $\mathbb{R}$ to $\mathbb{R}$ with period $1$.
Let $S$ be the skew product on $\mathbb{T}\times\Gamma\backslash G$ defined in \eqref{S}.
Then the M\"obius function $\mu$ is  linearly disjoint from $(\mathbb{T}\times\Gamma\backslash G, S)$.
\end{theorem}

\begin{theorem}\label{thm1}
Let $\mathbb{T}$ be the unit circle and $\Gamma\backslash G$ the $3$-dimensional Heisenberg nilmanifold.
Let $\alpha\in[0, 1)$, $k\in\mathbb{R}$ and $\phi, \psi$ be $C^\infty$-smooth periodic functions from $\mathbb{R}$ to $\mathbb{R}$ with period $1$ such that
\begin{equation}\label{1}
\int_0^1 \phi(t)dt=0.
\end{equation}
Let the skew product $T$ on $\mathbb{T}\times\Gamma\backslash G$ be given by \eqref{T}.
Then the M\"obius function $\mu$ is  linearly disjoint from $(\mathbb{T}\times\Gamma\backslash G, T)$.
\end{theorem}

\noindent\textbf{Notations.}
We list some notations that we use in this paper.
We write $e(x)$ for $e^{2\pi ix}$, and write $\|x\| $ for the distance between $x$ and the nearest integer, that is
\[
\|  x\| =\min_{n\in\mathbb{Z}}|x-n|.
\]
For positive $A, $ the notations $B =O(A)$ or $B\ll A$ mean that there exists a positive
constant $c$ such that $|B|\leq cA$.
If the constant $c$ depends on a parameter $b$,
we write $B =O_b (A)$ or $B\ll_b A$.
The notation $A\asymp B$ means that $A\ll B$ and $B\ll A$.
For a topological space $X, $ we use $C(X)$ to denote the set of all continuous complex-valued functions on $X$.
If $X$ is a smooth manifold and $r\geq 1$ is an integer, then we use $C^r(x)$ to denote the set of all $f\in C(X)$ that have continuous $r$-th derivatives.

\vskip 5mm

\section{Proof of Theorem \ref{thmrational} }
In this section, we prove Theorem\ref{thmrational}, where $\alpha$ is a rational number.

The Proposition 2.3 in \cite{HLW} shows that the $\mathbb{C}$-linear subspace spanned by two classes of functions $\mathcal{A}$ and $\mathcal{B}$ is dense in $C(\mathbb{T}\times \Gamma\backslash G)$.

For integers $m, j$ with $0\leq j\leq m-1$, the functions $\psi_{mj}$ and $\psi^*_{mj}$ on $G$ are defined by
\begin{align*}
&
\psi_{mj}\left(
    \footnotesize{
    \begin {pmatrix}
    1 & y & z\\
    0 & 1 & x\\
    0 & 0 & 1
    \end {pmatrix}
    }
    \right)
=e(mz+jx)\sum_{b\in\mathbb{Z}}e^{-\pi(y+b+\frac jm)^2}e(mbx),
\\
&
\psi^*_{mj}\left(
   \footnotesize{
   \begin {pmatrix}
   1 & y & z\\
   0 & 1 & x\\
   0 & 0 & 1
   \end {pmatrix}
   }
   \right)
=ie(mz+jx)\sum_{b\in\mathbb{Z}}e^{-\pi(y+b+\frac jm+\frac 12)^2}e\begin {pmatrix}
\frac 12\begin {pmatrix}
y+b+\frac jm
\end{pmatrix}
+\begin {pmatrix}
mbx
\end {pmatrix}
\end {pmatrix}.
\end{align*}
Then $\psi_{mj}$ and $\psi^*_{mj}$ are $\Gamma$-invariant so that they can be regarded as functions on $\Gamma\backslash G$.

Recall that there is a unique Borel probability measure on $\Gamma\backslash G$ that is invariant
under the right translations, and therefore $L^2(\Gamma\backslash G)$ can be defined.
Let $V_0$ be the subspace of  $L^2(\Gamma\backslash G)$ consisting of all functions $f\in L^2(\Gamma\backslash G)$ satisfying
$$
f\left( \Gamma g
\footnotesize{
\begin {pmatrix}
1 & 0 & z\\
0 & 1 & 0\\
0 & 0 & 1
\end {pmatrix}
}\right)
=f(\Gamma g)
$$
for any $g\in G$ and $z\in \mathbb{R}$.  Set $C_0=V_0\cap C(\Gamma\backslash G)$.

\begin{lemma}(\cite[Proposition 2.3]{HLW})\label{Prop2.3HLW}
Let $\mathcal{A}$ be the subset of $f\in\mathrm{C}(\mathbb{T}\times\Gamma\backslash G)$ satisfying
\[
f:
\left(
t, \Gamma
\footnotesize{
\begin {pmatrix}
1 & y & z\\
0 & 1 & x\\
0 & 0 & 1
\end {pmatrix}
}
\right)
\mapsto
e(\xi_1 t+\xi_2 x+\xi_3 y)
\psi\left(
\Gamma
\footnotesize{
\begin {pmatrix}
1 & y & z\\
0 & 1 & x\\
0 & 0 & 1
\end {pmatrix}
}
\right),
\]
where $\xi_1, \xi_2, \xi_3\in\mathbb{Z}$ and $\psi=\psi_{mj}, \bar{\psi}_{mj}, \psi^*_{mj}$ or $\bar{\psi}^*_{mj}$ for some $0\leqslant j \leqslant m-1$.
Here $\bar{\psi}_{mj}$ and $\bar{\psi}^*_{mj}$ stand for the complex conjugates of $\psi_{mj}$ and  $\psi^*_{mj} $, respectively.
Let $\mathcal{B}$ be subset of $f\in\mathrm{C}(\mathbb{T}\times\Gamma\backslash G)$ satisfying $f((t, \Gamma g))=f_1(t)f_2(\Gamma g)$ with $f_1\in\mathrm{C}(\mathbb{T})$ and $f_2\in\mathrm{C_0}(\Gamma\backslash G)$.
Then the $\mathbb{C}$-linear subspace spanned by $\mathcal{A}\cup \mathcal{B}$ is dense in $C(\mathbb{T}\times \Gamma\backslash G)$.
\end{lemma}

Following \cite{HLW}, we should separately consider the two cases, namely $f\in\mathcal{A}$ and $f\in\mathcal{B}$.
The  case $f\in\mathcal{A}$ will be handled by Fourier analysis and a classical result of Hua \cite{Hua}, which is a generalization of Davenport \cite{Davenport}.

\begin{lemma}(\cite{Hua}) \label{lem:1}
Let $f(x) = \alpha_dx^d + \alpha_{d-1}x^{d-1}+\cdots  + \alpha_1x + \alpha_0 \in \mathbb{R}[x]$ and  $0 \leq a < q$. Then, for arbitrary $A > 0$,
$$
\sum_{\substack{n\leq N  \\  n\equiv a \hskip-1.7mm\mod q}} \mu(n)e(f(n)) \ll_A\frac{N}{\log^A N},
$$
where the implied constant may depend on $A$, $q$ and $d$, but is independent of $\alpha_d, \dots , \alpha_0$.
\end{lemma}

For the case that $f\in\mathcal{A}$ we have the following proposition.


\begin{proposition}\label{prop:2}
Let $(\mathbb{T}\times\Gamma\backslash G, S)$ be as in Theorem \ref{thmrational} with $\alpha\in\mathbb{Q}\cap [0, 1). $
Let $\mathcal{A}$ be as above. Then, for any $(t_0, \Gamma g_0)\in\mathbb{T}\times\Gamma\backslash G, $ any $f\in\mathcal{A}$ and any $A>0$,
\[
\sum_{n\leq N} \mu(n) f(S^n(t_0, \Gamma g_0))\ll_A \frac N{\log^{A}N},
\]where the implied constant depends on $A$ and $\alpha$ only.
\end{proposition}

\begin{proof}
 For simplicity, we only consider a typical $f\in\mathcal{A}$ defined by
\begin{equation}\label{2}
f\left(
t, \Gamma
\footnotesize{
\begin {pmatrix}
1 & y & z\\
0 & 1 & x\\
0 & 0 & 1
\end {pmatrix}
}
\right)
=e(t+x+y+z)\sum_{l\in\mathbb{Z}}e^{-\pi(y+l)^2}e(lx).
\end{equation}
A general $f$ can be treated in the same way.

To compute $f(S^n(t_0, \Gamma g_0))$ via $\eqref{2}$, we define
\begin{align*}
  S_1(n;t) & =\sum_{l=0}^{n-1}\varphi_1(\alpha l+t),\qquad
  S_2(n;t)  =\sum_{l=0}^{n-1}\varphi_2(\alpha l+t), \\
  S_3(n;t) & =\sum_{l=0}^{n-1}\psi(\alpha l+t), \ \qquad
  S_4(n;t)  =\sum_{l=0}^{n-2}\varphi_2(\alpha l+t)\sum_{j=l+1}^{n-1}\varphi_1(\alpha j+t)
\end{align*}
for $n\geqslant1$, $t\in\mathbb{T}$ and set $S_1(0;t)=S_2(0;t)=S_3(0;t)=S_4(0;t)=S_4(1;t)=0$ for simplicity.
A straightforward calculation gives
\begin{equation}\label{3'}
  S^n:(t_0, \Gamma g_0)\mapsto(t_{0}+n\alpha, \Gamma g_n),
\end{equation}
where
\begin{equation}\label{4'}
  g_n:=g_0
  {\footnotesize{
  \begin {pmatrix}
 1 & S_2(n;t_0) & S_3(n;t_0)+S_4(n;t_0)\\
 0 & 1 & S_1(n;t_0)\\
 0 & 0 & 1
 \end {pmatrix}
 }}.
\end{equation}
Now write
\[
g_0=
\footnotesize{
\begin {pmatrix}
 1 & y_0 & z_0\\
 0 & 1 & x_0\\
 0 & 0 & 1
 \end {pmatrix}
 },
 \qquad
 g_n=
 \footnotesize{\begin {pmatrix}
 1 & y_n & z_n\\
 0 & 1 & x_n\\
 0 & 0 & 1
  \end {pmatrix},
  }
  \]
where we may assume  $x_0, y_0, z_0\in[0, 1)$ without loss of generality,  so that $\eqref{4'}$ gives
\begin{align*}
  x_n &=x_0+S_1(n;t_0), \\
  y_n &=y_0+S_2(n;t_0), \\
  z_n &=z_0+y_0S_1(n;t_0)+S_3(n;t_0)+S_4(n;t_0).
\end{align*}
Substituting $\eqref{4'}$ into $\eqref{3'}$ and combining with \eqref{2}, we obtain that
\begin{equation}\label{5'}
\begin{split}
f(S^n(t_0, \Gamma g_0)) &=f\left(
                           t_0+n\alpha, \Gamma
                          \footnotesize{
                           \begin {pmatrix}
                           1 & y_n & z_n\\
                           0 & 1 & x_n\\
                           0 & 0 & 1
                           \end {pmatrix}
                           }
                          \right)
                           \\
                         &=e(t_0+n\alpha+x_n+y_n+z_n)\sum_{l\in\mathbb{Z}}e^{-\pi(y_n+l)^2}e(lx_n).
\end{split}
\end{equation}

To analyze $\eqref{5'}$, we define $\omega:\mathbb{R}\times \mathbb{R}\rightarrow\mathbb{R}$ by
\[
\omega(u,v)=\sum_{l\in\mathbb{Z}}e^{-\pi(v+y_0+l)^2}e\big(l\big(u+x_0\big)\big).
\]
Then $\omega$ is an double periodic analytic function which is bounded on $\mathbb{R}\times \mathbb{R}$. 
With this function $\omega$ we can rewrite \eqref{5'} as
\begin{equation}\label{fSn'}
\begin{split}
f(S^n(t_0, \Gamma g_0)) =\rho&\omega\big(S_1(n;t_0), S_2(n;t_0)\big)\\
     & \times e\Big(n\alpha+(y_0+1)S_1(n;t_0)+S_2(n;t_0)+S_3(n;t_0) +  S_4(n;t_0)  \Big),
\end{split}
\end{equation}
where $\rho:=e(t_0+x_0+y_0+z_0)$. 

Recall that $\alpha\in\mathbb{Q}\cap[0, 1)$ in the present situation, so that we can write $\alpha=a/q$ with $0\leq a<q$ and $(a, q)=1. $
Thus for any periodic function $h$ with period $1, $ we have $h(l_1\alpha+t_0)=h(l_2\alpha+t_0)$ whenever $l_1\equiv l_2 \hskip-1.5mm\mod q$.
For any $0\leq b<q$ and any periodic function $h$ with period $1$, define
\[
\gamma(h, b)=\sum_{l=0}^{b-1}h(l\alpha+t_0)
\]
and set $\gamma(h)=\gamma(h, q)/q$.
For any $n\equiv b\hskip-1.5mm\mod q$,  write $n-1\equiv b'\hskip-1.5mm\mod q$, $0\leqslant b, b'<q$. Then
\begin{align*}
  S_1(n;t_0) &=(n-b)\gamma(\varphi_1)+\gamma(\varphi_1, b),  \\
  S_2(n;t_0) &=(n-b)\gamma(\varphi_2)+\gamma(\varphi_2, b),  \\
  S_3(n;t_0) &=(n-b)\gamma(\psi)+\gamma(\psi, b)
\end{align*}
and
\begin{align*}
  S_4(n;t_0) & =\sum_{l=0}^{n-2}\varphi_2(\alpha l+t_0)\sum_{j=l+1}^{n-1}\varphi_1(\alpha j+t_0)  \\
          & =\Big(    \frac{n-1-q-b'}{q}+ \cdots +1\Big)    S_1(q;t_0) S_2(q;t_0)
              +\sum_{l=0}^{b'-1}\varphi_2(\alpha l+t_0)\sum_{j=l+1}^{b'}\varphi_1(\alpha j+t_0)\\
            &\hskip6mm +(n-1-b')q^{-1} \sum_{l=0}^{q-1} \varphi_2(\alpha l+t_0)  \sum_{j=l+1}^{q+b'}\varphi_1(\alpha j+t_0)
\end{align*}
are all real-valued polynomials in $n$ of degree lease than $3$ with coefficients depending on $\alpha$, $b$, $b'$ and $m$.
It follows from this and \eqref{fSn'} that
$$
f(S^n(t_0, \Gamma g_0)) =\rho  \omega\big(S_1(n;t_0), S_2(n;t_0)\big)     e\big( P_{b,b'}(n)  \big),
$$
where $P_{b,b'}(n)$ is a real-valued polynomial in $n$ of degree  lease than $3$  with coefficients depending on $\alpha, b, b'$ and $m$, and then
\begin{equation}\label{6'}
\sum_{n\leq N}\mu(n)f(S^n(t_0, \Gamma g_0))
=\rho    \sum_{b=0}^{q-1}\sum_{\substack{n\leq N \\ n\equiv b \hskip-1.5mm\mod q}}             \mu(n)\tilde{\omega}(n) e(P(n, b)),
\end{equation}
where $\tilde{\omega}(n)=\omega\big(S_1(n;t_0), S_2(n;t_0)\big)$ is a bounded smooth weight.
Hence, by Lemma \ref{lem:1} we have
\begin{equation}\label{7'}
\sum_{\substack{n\leq N \\ n\equiv b \hskip-1.8mm\mod q}}\mu(n)e(P(n, b))\ll_A\frac N{\log^{A}N}
\end{equation}
for arbitrary $A>0$, where the implied constant depending on $q$ (hence on $\alpha$) and $A$ only.
Substituting this  to \eqref{6'} we obtain the desired estimate.
\end{proof}

The other case $f\in\mathcal{B}$ can be reduced to the case of skew products on $\mathbb{T}^3$ which is already known as   \cite[Corollary 3.2]{HLW}.

\begin{lemma}(\cite[Corollary 3.2]{HLW})\label{Cor3.2HLW}
 Let $\alpha\in \mathbb{R}$, and let $h_1, h_2: \mathbb{T}\rightarrow \mathbb{R}$ be $C^{\infty}$-smooth functions, $T: \mathbb{T}^3\rightarrow \mathbb{T}^3$ be given by
 $T(x,y,z)=(x+\alpha, y+h_1(x), z+h_2(x))$. Then the  M{\"o}bius Disjointness Conjecture holds for $(\mathbb{T}^3, T)$.
\end{lemma}

For the case that $f\in\mathcal{B}$ we have the following proposition.

\begin{proposition}\label{prop:1}
Let $\mathcal{B}\subset C(\mathbb{T}\times\Gamma\backslash G)$ as above. Let $S$ be as in Theorem \ref{thmrational} and let $f\in\mathcal{B}$. Then, for any $(t_0, \Gamma g_0)\in\mathbb{T}\times\Gamma\backslash G, $
\[
\lim_{N\rightarrow\infty}\frac 1N\sum_{n=1}^\mathrm{N}\mu(n)f(S^n(t_0, \Gamma g_0))=0.
\]
\end{proposition}

\begin{proof}
 Let $\tilde{S}:\mathbb{T}^3\rightarrow\mathbb{T}^3$ be given by
\[
\tilde{S}:(t, x, y)\mapsto(t+\alpha, x+\varphi_1(t), y+\varphi_2(t)),
\]
and  $\pi$ be the projection of $\mathbb{T}\times\Gamma\backslash G$ onto $\mathbb{T}^3$ given by
\[
\pi: \left(
t, \Gamma
 \footnotesize{
  \begin {pmatrix}
  1 & y & z\\
  0 & 1 & x\\
  0 & 0 & 1
  \end {pmatrix}
  }
\right)
\mapsto(t, x, y).
\]
Then we have $\pi\circ S=\tilde{S}\circ\pi, $ and hence $(\mathbb{T}^3, \tilde{S})$ is a topological factor of $(\mathbb{T}\times\Gamma\backslash G, S)$.

Since $f\in\mathcal{B}$, there are some $f_1\in C(\mathbb{T})$ and $f_2\in V_0\cap C(\Gamma\backslash G)$ such that $f(t, \Gamma g)=f_1(t)f_2(\Gamma g)$.
Then for any $z'\in\mathbb{R}$,
\[
f_2\left(
\Gamma
\footnotesize{
   \begin {pmatrix}
   1 & y & z\\
   0 & 1 & x\\
   0 & 0 & 1
   \end {pmatrix}
   }
\right)
=f_2\left(
\Gamma
\footnotesize{
\begin {pmatrix}
1 & y & z+z'\\
0 & 1 & x\\
0 & 0 & 1
\end {pmatrix}
}
\right)
\]
which shows that  $f_2$ is independent of the $z$-component and induces a well-defined continuous function $\tilde{f_2}\in C(\mathbb{T}^2)$ given by
\[
\tilde{f_2}(x, y)=f_2
\left(
\Gamma
\footnotesize{
\begin {pmatrix}
1 & y & z\\
0 & 1 & x\\
0 & 0 & 1
\end {pmatrix}
}
\right)
\]
for any $z\in\mathbb{R}. $ Define $\tilde{f}(t, x, y)\in C(\mathbb{T}^3)$ by
\[
\tilde{f}(t, x, y)=f_1(t)\tilde{f_2}(x, y).
\]
Then $f(t, \Gamma g)=\tilde{f}\circ\pi(t, \Gamma g)$ for any $(t, \Gamma g)\in\mathbb{T}\times\Gamma\backslash G. $ Hence
\[
f(S^n(t_0, \Gamma g_0))=\tilde{f}\circ\pi\circ S^n(t_0, \Gamma g_0)=\tilde{f}\circ \tilde{S}^n\circ\pi(t_0, \Gamma g_0)
\]
for any $n\geq1$, and the sequence $\{f(S^n(t_0, \Gamma g_0))\}_{n\geq1}$ is  observed in $(\mathbb{T}^3, \tilde{S})$.
The desired result follows from Lemma \ref{Cor3.2HLW}.
\end{proof}

\noindent\emph{Proof of Theorem \ref{thmrational}}.\
Combining Propositions \ref{prop:2} and \ref{prop:1} we get that Theorem \ref{thmrational} holds.\qed

\section{Measure complexity}

To prove Theorem \ref{thm1} for irrational $\alpha$, we will use the criterion given by Huang-Wang-Ye in \cite{HWY}.
To do this, we need the  concept of measure complexity.
In this section, we will collect some concepts and facts from  \cite{HWY} without proof.

Let $(X, T)$ be a flow, and $M(X, T)$ the set of all $T$-invariant Borel probability measures on $X$.
A metric $d$ on $X$ is said to be compatible if the topology induced by $d$ is the same as the given topology on $X$. For a compatible metric $d$ and an $n\in\mathbb{N}, $ define
\begin{equation}\label{bard}
\bar{d}_n(x, y)=\frac 1n\sum_{j=0}^{n-1}d(T^{j}x, T^{j}y)
\end{equation}
for $x, y\in X. $ Then, for $\varepsilon>0$, let
\begin{equation*}
B_{\bar{d}_n}(x, \varepsilon)=\{y\in X:\bar{d}_n(x, y)<\varepsilon\},
\end{equation*}
with which we can further define, for $\rho\in M(X, T), $
$$
     s_n(X, T, d, \rho, \varepsilon)
     =\min\{m\in\mathbb{N}:\exists x_1, . . . , x_m\in X,  \ \textrm{s. t.}\  \rho(\cup_{j=1}^{m}B_{\bar{d}_n}(x_j, \varepsilon))>1-\varepsilon \}.
$$
Let $(X, d, T, \rho)$ be as above, and let $\{u(n)\}_{n\geq 1}$ be an increasing sequence satisfying $1\leq u(n)\rightarrow \infty$ as $n\rightarrow\infty. $ We say that the measure complexity of $(X, d, T, \rho)$ is weaker than $u(n)$ if
\begin{equation*}
\liminf_{n\rightarrow\infty}\frac{s_n(X, T, d, \rho, \varepsilon)}{u(n)}=0
\end{equation*}
for any $\varepsilon>0$. In view of Lemma \ref{HWYprop2.2} this property is independent of the
choice of compatible metrics. Hence we can say instead that the measure complexity
of $(X, T, \rho)$ is weaker than $u(n). $ We say the measure complexity of $(X, T, \rho)$ is sub-polynomial if the measure complexity of $(X, T, \rho)$ is weaker than $n^\tau$ for any $\tau>0$.

We will use the  equivalent condition provided by Huang-Wang-Ye \cite{HWY} to derive that the M{\"o}bius Disjointness Conjecture holds for our flow $(\mathbb{T}\times\Gamma\backslash G, T)$ as in Theorem \ref{thm1} for $\alpha$ irrational.
The following lemma is the main theorem of \cite{HWY} which gives a criterion of the M{\"o}bius Disjointness Conjecture.
\begin{lemma}(\cite[Theorem 1.1]{HWY})\label{HWY}
  If the measure complexity of $(X, T, \rho)$ is sub-polynomial for any $\rho\in M(X; T)$,
  then the M{\"o}bius Disjointness Conjecture holds for $(X, T)$.
\end{lemma}

Let $(X, T)$ and $(Y, S)$ be two flows, $d$ and $d'$  the metrics on $X$ and $Y$, respectively.
Let $\rho\in M(X, T)$, $\nu \in M(Y, S)$, $\mathcal{B}_X$ and $\mathcal{B}_Y$ be the Borel $\sigma$-algebras of $X$ and $Y$, respectively.
We say $(X, \mathcal{B}_X, T, \rho)$ is measurably isomorphic to $(Y, \mathcal{B}_Y, S, \nu)$,
if there exist $X'\subset X$, $Y'\subset Y$ with $\rho(X') = \nu(Y') = 1$, $TX' \subset X'$, $SY'\subset Y'$,
and an invertible measure-preserving map $\theta: X'\rightarrow Y'$
such that $\theta\circ T(x) = S\circ \theta(x)$ for any $x\in X'$.
The following lemma is important when calculating the measure complexity.

\begin{lemma} (\cite[Proposition 2.2]{HWY})\label{HWYprop2.2}\label{HLW5.2}
Let $\{u(n)\}_{n\geq1}$ be an increasing sequence satisfying $1\leq u(n)\rightarrow \infty$ as $n\rightarrow 1$.
Assume that $(X, \mathcal{B}_X, T, \rho)$ is measurably isomorphic to $(Y, \mathcal{B}_Y, S, \nu)$.
Then the measure complexity of $(X, d, T, \rho)$ is weaker than $u(n)$ if and only if the measure
complexity of $(Y, d', S, \nu)$ is weaker than $u(n)$.
\end{lemma}

Now we need to choose a proper metric on  $\mathbb{T}\times\Gamma\backslash G$.
The following facts can be found in Sections 2 and 5 in Green-Tao \cite{GT}, which we will
directly state without proof.

The lower central series filtration $G_\bullet$ on $G$ is the sequence of closed connected subgroups
$$
G=G_1\supseteq G_2\supseteq G_3=\{\textup{id}_G\},
$$
where
\begin{center}
$G_2=[G_1,G]=\begin{pmatrix}\begin{smallmatrix} 1&0&\mathbb{R}\\0&1&0\\0&0&1
\end{smallmatrix}\end{pmatrix}$
\end{center}
and id$_G$ is the identity element of $G$. Let $\mathfrak{g}$ be the Lie algebra of $G$. Let
$$
X_1=\begin{pmatrix}\begin{smallmatrix} 0&0&0\\0&0&1\\0&0&0\end{smallmatrix}\end{pmatrix}, \qquad
X_2=\begin{pmatrix}\begin{smallmatrix} 0&1&0\\0&0&0\\0&0&0\end{smallmatrix}\end{pmatrix},\qquad
X_3=\begin{pmatrix}\begin{smallmatrix} 0&0&1\\0&0&0\\0&0&0\end{smallmatrix}\end{pmatrix}.
$$
Then  $\mathcal{X}=\{X_1,X_2,X_3\}$ is a Mal'cev basis adapted to $G_\bullet$.
The corresponding Mal'cev coordinate map $\kappa:G\rightarrow \mathbb{R}^3$ is given by
\begin{equation}\label{HLW6.1}
\kappa
\left(
\footnotesize{
\begin{pmatrix}\begin{smallmatrix} 1&y&z\\0&1&x\\0&0&1\end{smallmatrix}\end{pmatrix}
}\right)
=(x,y,z-xy)
\end{equation}
and the metric $d_G$ on $G$ is defined by
\begin{center}
$d_G(g_1,g_2):=\inf\Big\{\sum\limits_{i=0}^{n-1}\min(|\kappa(h^{-1}_{i-1}h_i)|,
|\kappa(h^{-1}_{i}h_{i-1})|)\ :\ h_0,\dots,h_n\in G, h_0=g_1,h_n=g_2\Big\}$
\end{center}
from which we see that $d_G$ is left-invariant. By \eqref{HLW6.1}, we can get
\begin{equation}\label{HLW6.2}
\left|
\kappa
\left(\footnotesize{\begin{pmatrix}\begin{smallmatrix} 1&y&z\\0&1&x\\0&0&1\end{smallmatrix}\end{pmatrix}
}\right)\right|
\leq |x|+|y|+|z|
\end{equation}
with $x,y\in [0,1)$.
The above metric on $G$ descends to a metric on $\Gamma\backslash G$ given by
\begin{center}
$d_{\Gamma\backslash G}(\Gamma g_1,\Gamma g_2):=\inf\big\{d_G(g_1',g_2')\ : \ g_1',g_2'\in G, \Gamma g_1=\Gamma g_1',\Gamma g_2=\Gamma g_2'\big\}$.
\end{center}
It can be proved that $d_{\Gamma\backslash G}$ is indeed a metric on  $\Gamma\backslash G$.
Since $d_G$ is left-invariant,
\begin{equation}\label{HLW6.3}
d_{\Gamma\backslash G}(\Gamma g_1,\Gamma g_2)=\inf_{\gamma\in\Gamma}d_G(g_1,\gamma g_2).
\end{equation}
Finally,  take $d_\mathbb{T}$ to be the canonical Euclidean metric on $\mathbb{T}$, and $d = d_{\mathbb{T}\times \Gamma\backslash G}$ the
$l^{\infty}$-product metric of $d_{\mathbb{T}}$ and $d_{ \Gamma\backslash G}$  given by
\begin{equation}\label{HLW(6.4)}
d\big((t_1, \Gamma g_1), (t_2, \Gamma g_2)\big) = \max\big(d_{\mathbb{T}}(t_1, t_2),  d_{\Gamma \backslash G}(\Gamma g_1, \Gamma g_2)\big).
\end{equation}
In view of Lemma \ref{HLW5.2}, the choice of compatible metrics does not affect the measure complexity. Thus the above choice of $d$ is admissible.

\section{Proof of Theorem \ref{thm1}}
In this section, we assume that $\alpha$  is irrational and prove Theorem \ref{thm1} for irrational $\alpha$.
Then combine with Theorem \ref{thmrational} we get
 Theorem \ref{thm1}.

Let
\[
\alpha=[0; a_1, a_2, \dots , a_i, \dots ]=\frac {1}{a_1+\frac {1}{a_2+\frac 1{a_3+\cdots }}}
\]
be the continued fraction expansion of $\alpha$.
Then the expansion is infinite since $\alpha$ is irrational.
Let $l_i/q_i=[0; a_1, a_2, \dots , a_i]$ be the $i$-th convergent of $\alpha$.
Let $\mathcal{Q}=\{q_i:i\geq1\}. $ For $B>2, $ define
\[
\mathcal{Q}^{\flat}(B)=\{q_i\in\mathcal{Q}:q_{i+1}\leq q_i^B\}\cup\{1\}
,\qquad
\mathcal{Q}^{\sharp}(B)=\{q_i\in\mathcal{Q}:q_{i+1}>q_i^B>1\}.
\]
Furthermore, we define
\[
M_1(B)=\bigcup_{q_i\in\mathcal{Q}^{\sharp}(B)}\{m\in\mathbb{Z}:q_i\leq|m|<q_{i+1}, q_i| m\}\cup\{0\}
\]
and define $M_2(B)=\mathbb{Z}\backslash M_1(B)$. Now expand $\phi$ into Fourier series
\[
\phi(t)=\sum_{m\in\mathbb{Z}}\hat{\phi}(m)e(mt),
\]
and further decompose $\phi$ as
\begin{equation*}
  \begin{split}
  \phi(t) &=\phi_1(t)+\phi_2(t)  \\
    &:=\sum_{m\in M_1(B)}\hat{\phi}(m)e(mt)+\sum_{m\in M_2(B)}\hat{\phi}(m)e(mt).
\end{split}
\end{equation*}
We call $\phi_1$ and $\phi_2$ the resonant and non-resonant part of $\phi, $ respectively. Let $\eta(t)=\phi^2(t). $ We can do the same decompositions for $\eta$ and $\psi, $ getting
\begin{equation}\label{8}
\eta(t)=\eta_1(t)+\eta_2(t), \qquad \psi(t)=\psi_1(t)+\psi_2(t).
\end{equation}
Note that the above decompositions of $\phi$, $\eta$ and $\psi$ depend on the parameter $B$, though we do not
make it explicit.

In our proof of Theorem \ref{thm1} for irrational $\alpha$ we still need a lemma from \cite{HLW}.

\begin{lemma}(\cite[Lemma 4.2]{HLW})\label{HLW4.2}
Let $B > 2$ and let $\{a(m)\}_{m\in \mathbb{Z}}$ be a sequence such that $|a(m)|\ll m^{-2B}$.
Then the series
$$
\sum_{m\in M_2(B)}\frac{a(m)}{e(m\alpha)-1}
$$
is absolutely convergent.
\end{lemma}

Define
\begin{equation}\label{gvarphi}
 g_\phi (t)=\sum_{m\in M_2(B)}\hat{\phi}(m)\frac {e(mt)}{e(m\alpha)-1}.
\end{equation}
Since $\phi$ is assumed to be $C^{\infty}$-smooth, we have $\hat{\phi}(m) \ll |m|^{-2B}$  for any $B > 0$.
Therefore, by Lemma \ref{HLW4.2},  $g_\phi$ is a continuous periodic function with period $1$ and satisfies
\begin{equation}\label{varphi_2}
g_\phi (t+\alpha)-g_\phi (t)=\phi_2(t).
\end{equation}
Similarly, there exist continuous periodic functions $g_\eta$ and $g_\psi$ with period $1$ such that
\begin{equation}\label{eta_2psi_2}
  \eta_2(t)=g_\eta (t+\alpha)-g_\eta (t), \qquad
\psi_2(t)=g_\psi (t+\alpha)-g_\psi (t).
\end{equation}

To investigate the resonant part, define
\begin{equation}\label{9}
\Phi_n(t)=\sum_{l=0}^{n-1}\phi_1(l\alpha+t), \quad  H_n(t)=\sum_{l=0}^{n-1}\eta_1(l\alpha+t), \quad \Psi_n(t)=\sum_{l=0}^{n-1}\psi_1(l\alpha+t)
\end{equation}
for $n\in\mathbb{N}$ and $t\in\mathbb{T}$. For $n=q_i$, Huang-Liu-Wang \cite{HLW} proved the following proposition.

\begin{lemma}(\cite[Lemma 4.3]{HLW})\label{HLW4.3}
Let $B> 2$. Then there exists a positive constant $C_1 = C_1(B)$ depending on $B$ only, such that the three inequalities
$$
|\Phi_{q_i}(t)-q_i\hat{\phi}(0)|\leq C_1q_i^{-B+1};  \ \
|H_{q_i}(t)-q_i\hat{\eta}(0)|\leq C_1q_i^{-B+1};  \ \
|\Psi_{q_i}(t)-q_i\hat{\psi}(0)|\leq C_1q_i^{-B+1}
$$
hold simultaneously for all $t\in \mathbb{T}$ and all $q_i\in \mathbb{Q}^\sharp(B)$.
\end{lemma}

Now we are ready to prove Theorem \ref{thm1} for irrational $\alpha$ and the skew product $T$ defined in \eqref{T}. We will prove the following proposition.

\begin{proposition}\label{prop:4}
Let $(\mathbb{T}\times\Gamma\backslash G, T)$ be as in Theorem \ref{thm1} with $\alpha$ irrational.
Then the measure complexity of $(\mathbb{T}\times\Gamma\backslash G, T)$ is sub-polynomial for any $\rho\in M(\mathbb{T}\times\Gamma\backslash G, T)$.
\end{proposition}

\begin{proof}
Fix $\tau>0. $ We want to show that, for any $\varepsilon>0, $
\[
\liminf_{n\rightarrow\infty}\frac {s_n(\mathbb{T}\times\Gamma\backslash G, T, d, \rho, \varepsilon)}{n^\tau}=0.
\]
Without loss of generality, we assume that both $\tau$, $\varepsilon< 10^{-2}$,
and $\tau^{-1}$, $\varepsilon^{-1}$ are integers.
Set $B=8\tau^{-1}+1$.

Firstly, assume that $\mathcal{Q}^\sharp(B)$ is infinite.
Construct a transformation $R:\mathbb{T}\times\Gamma\backslash G\rightarrow\mathbb{T}\times\Gamma\backslash G$ as
\begin{equation}\label{10}
R:(t, \Gamma g)\mapsto
\left(
t, \Gamma g
\footnotesize{
\begin {pmatrix}
1 & k g_{\phi}(t) & \frac 12 k g_{\phi}^2(t)-\frac 12 k g_{\eta}(t)+g_{\psi}(t)\\
0 & 1 & g_{\phi}(t)\\
0 & 0 & 1
\end{pmatrix}
}
\right).
\end{equation}
Write $T_1=R^{-1}\circ T \circ R. $ Then a straightforward calculation gives
\begin{equation*}
  \begin{split}
     T_1(t, \Gamma g) &=R^{-1}\circ T \circ R(t, \Gamma g) \\
       &=R^{-1}\circ T
\left(
t, \Gamma g
\footnotesize{
\begin{pmatrix}
1 & k g_{\phi}(t) & \frac 12 k g_{\phi}^2(t)-\frac 12 k g_{\eta}(t)+g_{\psi}(t)\\
0 & 1 & g_{\phi}(t)\\
0 & 0 & 1
\end{pmatrix}
}
\right)  \\
&=R^{-1}\hskip-1mm
\left(\hskip-1mm
t+\alpha, \Gamma g
\footnotesize{
\begin{pmatrix}
1 & k\phi(t)+k g_{\phi}(t) & \psi(t)+k g_{\phi}(t)\phi(t) +  \frac 12 k g_{\phi}^2(t)-\frac 12 k g_{\eta}(t)+g_{\psi}(t)\\
0 & 1 & \phi(t)+g_{\phi}(t)\\
0 & 0 & 1
\end{pmatrix}
}\hskip-1mm
\right)
\\
&=\left(
t+\alpha, \Gamma g
\footnotesize{
\begin {pmatrix}
1 & k\phi(t)+k g_{\phi}(t)-k g_{\phi}(t+\alpha) & \omega\\
0 & 1 & \phi(t)+g_{\phi}(t)-g_{\phi}(t+\alpha)\\
0 & 0 & 1
\end{pmatrix}
}
\right),
\end{split}
\end{equation*}
where
\begin{equation*}
\begin{split}
\omega &=\frac 12 k g_{\phi}^2(t+\alpha)+\frac 12 k g_{\phi}^2(t)-k g_\phi(t)g_\phi(t+\alpha)+\frac 12 k g_\eta(t+\alpha) \\
    &\hskip8mm -g_\psi(t+\alpha)-k\phi(t)g_\phi(t+\alpha)+\psi(t)+k\phi(t)g_\phi(t)-\frac 12 k g_\eta(t)+g_\psi(t).
\end{split}
\end{equation*}
So $ \omega$ can be simplified as
  \begin{equation}\label{11}
  \begin{split}
     \omega &=\frac 12 k (g_\phi(t+\alpha)-g_\phi(t))^2+\frac 12 k (g_\eta(t+\alpha)-g_\eta(t))\\
     &\hskip46mm -k\phi(t)(g_\phi(t+\alpha)-g_\phi(t))+\psi_1(t) \\
       &=\frac 12 k\phi_2^2(t)+\frac 12 k\eta_2(t)-k\phi(t)\phi_2(t)+\psi_1(t)\\
       &=\frac 12 k(\phi(t)-\phi_2(t))^2-\frac 12 k \phi^2(t)+\frac 12 k\eta_2(t)+\psi_1(t)\\
     &=\frac 12 k\phi_1^2(t)-\frac 12 k\eta_1(t)+\psi_1(t),
   \end{split}
  \end{equation}
where we have used \eqref{gvarphi}, \eqref{eta_2psi_2} and \eqref{varphi_2}.
It follows that
\begin{equation}\label{12}
T_1:(t, \Gamma g)\mapsto
\left(
t+\alpha, \Gamma g
\footnotesize{
\begin{pmatrix}
1 & k\phi_1(t) & \frac 12 k\phi_1^2(t)-\frac 12 k\eta_1(t)+\psi_1(t)\\
0 & 1 & \phi_1(t)\\
0 & 0 & 1
\end{pmatrix}
}
\right)
\end{equation}
and then
\begin{equation*}
  T_1^n:(t, \Gamma g)\mapsto
\left(
t+n\alpha, \Gamma g
\footnotesize{
\begin {pmatrix}
1 & k\Phi_n(t) & \frac 12 k\Phi_n^2(t)-\frac 12 k H_n(t)+\Psi_n(t)\\
0 & 1 & \Phi_n(t)\\
0 & 0 & 1
\end{pmatrix}
}
\right).
\end{equation*}
Clearly, $R$ is a homeomorphism on $\mathbb{T}\times\Gamma\backslash G$. Hence by Lemma \ref{HWYprop2.2}, we need only to show that the measure complexity of $(\mathbb{T}\times\Gamma\backslash G, T_1, v)$ is weak than $n^\tau$, where $v=\rho\circ R$.

Let $C_1 = C_1(B) > 0$ be the constant in Lemma \ref{HLW4.3}. The functions $\phi_1(t), $ $\eta_1(t)$ and $\psi_1(t)$ are Lipschitz continuous, and therefore there exists $L>0$ such that
\begin{equation}\label{13}
\begin{split}
&|\phi_1(t_1)-\phi_1(t_2)|\leq L\|  t_1-t_2\|,
\\
&|\eta_1(t_1)-\eta_1(t_2)|\leq L\|  t_1-t_2\|,
\\
&|\psi_1(t_1)-\psi_1(t_2)|\leq L\|  t_1-t_2\|
\end{split}
\end{equation}
for any $t_1, t_2\in\mathbb{T}. $ We also assume that $L$ is large enough such that $L>\varepsilon^{-1}. $ Moreover, since $\phi_1(t), $ $\eta_1(t)$ and $\psi_1(t)$ are continuous, there exists a constant $C_2>0$ such that
\[
|\phi_1(t)|\leq C_2, \qquad|\eta_1(t)|\leq C_2, \qquad|\psi_1(t)|\leq C_2
\]
for all $t\in\mathbb{T}. $ Since $q_i\rightarrow\infty$ as $i\rightarrow\infty$,
there exists a constant $I_0>0$ such that $(C_1+C_2)/q_i<\varepsilon$ for all $i\geq I_0$.
For $i\geq I_0$, define
\[
F_1(i)=\Big\{t=\frac {j\varepsilon}{Lq_i}\in\mathbb{T}\ :\ j=0, 1, \dots , \frac {Lq_i}{\varepsilon}-1 \Big\}
\]
and
\[
F_2(i)=\left\{\Gamma g=
\footnotesize{
\begin {pmatrix}
1 & j_2(q_i^2 L)^{-1} & j_3(q_i^2 L)^{-1}\\
0 & 1 & j_1(q_i^2 L)^{-1}\\
0 & 0 & 1
\end{pmatrix}
}
\in\Gamma\backslash G\ :\ j_1, j_2, j_3 =0, 1, \dots  , q_i^2 L-1\right\}.
\]
Let
\[
F(i)=\{(t, \Gamma g)\in\mathbb{T}\times\Gamma\backslash G\ :\ t\in F_1(i), \;\Gamma g\in F_2(i)\}.
\]
Then $\sharp F(i)=\varepsilon^{-1}L^4 q_i^7. $

Now assume that $q_i\in\mathcal{Q}^\sharp(B)$ with $i\geq I_0 $ and set
\begin{equation}\label{14}
  n_i=q_i^{B-1}.
\end{equation}
Then any positive integer $m\leq n_i$ can be uniquely written as
\begin{equation}\label{15}
  m=a_m q_i+b_m
\end{equation}
with $0\leq b_m<q_i$ and $a_m\leq q_i^{B-2}. $ By the definition of $F(i), $ for any
\[
(t, \Gamma g)=
\left(
t, \Gamma
\footnotesize{
\begin {pmatrix}
1 & y & z\\
0 & 1 & x\\
0 & 0 & 1
\end {pmatrix}
}
\right)
\in\mathbb{T}\times\Gamma\backslash G
\]
with $x, y, z\in[0, 1), $ there exists
\[
(t^*, \Gamma g^*)=
\left(
t^*, \Gamma
\footnotesize{
\begin {pmatrix}
1 & y^* & z^*\\
0 & 1 & x^*\\
0 & 0 & 1\end {pmatrix}
}
\right)
\in F(i)
\]
\\such that $\|t-t^*\|\leq\varepsilon/(Lq_i)$ and
\begin{equation}\label{16}
  \max\{|x-x^*|, \ |y-y^*|,\  |z-z^*|\}\leq\frac 1{q_i^2 L}.
\end{equation}
We want to show that $d(T_1^m(t, \Gamma g), T_1^m(t^*, \Gamma g^*))$ is small for any $m\leq n_i$,
 where $n_i$ is as in $\eqref{14}$.
 Let
\[
Y(m)=
\footnotesize{
\begin{pmatrix}
       1 & k\Phi_m(t) & \frac 12 k \Phi_m^2(t)-\frac 12 k H_m(t)+\Psi_m(t)\\
       0 & 1 & \Phi_m(t) \\
       0 & 0 & 1
     \end{pmatrix}
}
\]
and

\[
Y^*(m)=
\footnotesize{
\begin{pmatrix}
       1 & k\Phi_m(t^*) & \frac 12 k \Phi_m^2(t^*)-\frac 12 k H_m(t^*)+\Psi_m(t^*)\\
       0 & 1 & \Phi_m(t^*) \\
       0 & 0 & 1
     \end{pmatrix}
}.
\]
Then we have
\[
T_1^m(t, \Gamma g)=(t+\alpha m, \Gamma gY(m)), \quad T_1^m(t^*, \Gamma g^*)=(t^*+\alpha m, \Gamma g^*Y^*(m)).
\]
Therefore, by our choice of the metric on $\mathbb{T}\times\Gamma\backslash G$ in \eqref{HLW(6.4)}, we have
\begin{equation}\label{17}
  d\big(T_1^m(t, \Gamma g), T_1^m(t^*, \Gamma g^*)\big)
  \leq
  \max\big(\|t-t^*\|, d_{\Gamma\backslash G}(\Gamma gY(m), \Gamma g^*Y^*(m))\big).
\end{equation}
Since the term $\|t-t^*\|$ can be arbitrarily small as $q_i\rightarrow\infty $, it remains only to bound the last term in \eqref{17}.
 By the triangle inequality and \eqref{HLW6.3} we get
\begin{align} \label{18}
d_{\Gamma\backslash G}(\Gamma gY(m), \Gamma g^*Y^*(m)) &\leq d_{\Gamma\backslash G}(\Gamma g^*Y(m), \Gamma gY(m))+d_{\Gamma\backslash G}(\Gamma g^*Y^*(m), \Gamma g^*Y(m)) \nonumber \\
&\leq d_G(g^*Y(m), gY(m))+d_G(g^*Y^*(m), g^*Y(m))\\
           & =d_G(g^*Y(m), gY(m))+d_G(Y^*(m), Y(m))\nonumber ,
\end{align}
where the last equality follows from the left invariance of $d_G$.
Furthermore, from the definition of $d_G$ we get
\begin{equation}\label{19}
 \begin{split}
    d_G(g^*Y(m), gY(m)) & \leq|\kappa(Y(m)^{-1}g^{-1}g^*Y(m))|, \\
    d_G(Y^*(m), Y(m))   &\leq|\kappa(Y(m)^{-1}Y^*(m))|,
 \end{split}
\end{equation}
where $\kappa$ is the Mal'cev coordinate map \eqref{HLW6.1}, and $|\cdot |$ is the $l^\infty$-norm on $\mathbb{R}^3. $

A straightforward calculation gives
\begin{align*}
Y(m)^{-1}&g^{-1}g^* Y(m)   \\
&=
\footnotesize{
\begin{pmatrix}
       1 & -k\Phi_{m}(t) & \frac 12 k\Phi_m^2(t)+\frac 12 k H_m(t)-\Psi_m(t) \\
       0 & 1 & -\Phi_m(t) \\
       0 & 0 & 1
     \end{pmatrix}
     \begin{pmatrix}
       1 & y^*-y & z^*-z+xy-x^*y \\
       0 & 1 & x^*-x \\
       0 & 0 & 1
     \end{pmatrix}
}
\\
&\hskip56mm \times
\footnotesize{
\begin{pmatrix}
       1 & k\Phi_{m}(t) & \frac 12 k\Phi_m^2(t)-\frac 12 k H_m(t)+\Psi_m(t) \\
       0 & 1 & \Phi_m(t) \\
       0 & 0 & 1
     \end{pmatrix}
}
\\
&=
\footnotesize{
\begin{pmatrix}
1 & y^*-y & z^*-z+y(x-x^*)+\Phi_m(t)(kx-kx^*-y+y^*)\\
0 & 1 & x^*-x\\
0 & 0 & 1
\end{pmatrix}
}.
\end{align*}
Since $x, y\in[0, 1)$, from  \eqref{HLW6.2} we get
\begin{equation*}
| \kappa(Y(m)^{-1}g^{-1}g^* Y(m))|\leq
\big(\max(|k|, 1)|\Phi_m(t)|+2\big)(| x-x^*|+| y-y^*|+| z-z^*|).
\end{equation*}
By Lemma \ref{HLW4.3}, and noticing  that $\hat{\phi}(0)=0$ as in \eqref{1},  we get
\[
|\Phi_{q_i}|\leq C_1 q_i^{-B+1}.
\]
Hence by the definition of $\Phi_n(t)$ and $\eqref{15}$, we obtain
\begin{equation*}
  \begin{split}
     |\Phi_m(t)| &\leq\sum_{r=0}^{a_m-1}|\Phi_{q_i}(t+rq_i \alpha)|+\sum_{l=0}^{b_m}|\phi_1(t+(a_m q_i +l)\alpha)|
     \leq \frac {C_1 a_m}{q_i^{B-1}}+C_2 q_i \leq\frac {C_1}{q_i}+C_2q_i.
   \end{split}
\end{equation*}
Thus by $\eqref{16}, $ $\eqref{18}$ and $\eqref{19}$, 
we obtain
\begin{equation}\label{20}
  \begin{split}
     d_G(g^*Y(m), gY(m))&\leq
     \big(\max(|k|, 1)|\Phi_m(t)|+2\big)(| x-x^*|+| y-y^*|+| z-z^*|)
\\
      &\leq \big(\max(|k|, 1)(1+C_2q_i)+2\big)\frac{3}{q_i^2L}    \leq 6\big(\max(|k|, 1) + 1 \big)\varepsilon.
  \end{split}
\end{equation}
The treatment of $d_G(Y^*(m), Y(m))$ is similar. We calculate that
\begin{align*}
Y^{-1}(m)  Y^*(m)
& =
\footnotesize{
\begin{pmatrix}
       1 & -k\Phi_{m}(t) & \frac 12 k\Phi_{m}^2(t)+\frac 12 k H_{m}(t)-\Psi_{m}(t) \\
       0 & 1 & -\Phi_{m}(t) \\
       0 & 0 & 1
     \end{pmatrix}
}
\\
&\hskip6mm \times
\footnotesize{
     \begin{pmatrix}
        1 & k\Phi_{m}(t^*) & \frac 12 k\Phi_{m}^2(t^*)-\frac 12 kH_{m}(t^*)+\Psi_{m}(t^*) \\
       0 & 1 & \Phi_{m}(t^*) \\
       0 & 0 & 1
     \end{pmatrix}
}
\\
&=
\footnotesize{
\begin{pmatrix}
       1 & k(-\Phi_m(t)+\Phi_m(t^*)) & \tilde{\omega}  \\
       0 & 1 & -\Phi_m(t)+\Phi_m(t^*) \\
       0 & 0 & 1
     \end{pmatrix}
},
\end{align*}
where
\[
\tilde{\omega} =\frac 12 k(\Phi_m(t^*)-\Phi_m(t))^2 +\frac 12 k(H_m(t)-H_m(t^*))+(\Psi_m(t^*)-\Psi_m(t)).
\]

By Lemma \ref{HLW4.3}, \eqref{13} and \eqref{15}, we get
\begin{align*}
 |\Phi_m(t^*)-\Phi_m(t)| &\leq\sum_{r=0}^{a_m-1}\Big(  |\Phi_{q_i}(t^*+rq_i \alpha)|+|\Phi_{q_i}(t+rq_i \alpha)| \Big)\\
                         &\hskip8mm+     \sum_{l=0}^{b_m}    |\phi_1(t^*+(a_m q_i +l)\alpha)-\phi_1(t+(a_m q_i +l)\alpha)| \\
                         &\leq \frac {C_1}{q_i}+q_iL\|t^*-t \| <2\varepsilon.
\end{align*}
Similarly, we can also obtain  $| H_m(t^*)- H_m(t)|\leq 2\varepsilon$ and $|\Psi_m(t^*)-\Psi_m(t)|\leq 2\varepsilon$.
Therefore, by the definition of $d_G$ and \eqref{HLW6.2}, we get
\begin{equation}\label{21}
  d_G(Y^*(m), Y(m))\leq|\kappa Y(m)^{-1}Y^*(m)|<   8 \max(|k|, 1) \varepsilon.
\end{equation}

From $\eqref{17}, $ $\eqref{18}, $ $\eqref{20}$ and $\eqref{21}, $ we conclude that
\[
d\big(T_1^m(t, \Gamma g), T_1^m(t^*, \Gamma g^*)\big)\leq  \big(14\max(|k|, 1)+6\big)\varepsilon
\]
for all $m\leq n_i$. Here, and in what follows, $n_i$ is as in \eqref{14}.
So, by the definition of $\bar{d}_n$ in \eqref{bard}, we have
\begin{align*}
\bar{d}_{n_i}((t, \Gamma g), (t^*, \Gamma g^*))
&=\frac 1{n_i}\sum_{m=0}^{n_i -1}d(T_1^m(t, \Gamma g), T_1^m(t^*, \Gamma g^*))
\leq \big(14\max(|k|, 1)+6\big)\varepsilon.
\end{align*}
This means that $\mathbb{T}\times\Gamma\backslash G$ can be covered by $\sharp F(k)=\varepsilon^{-1}L^4 q_i^7$ balls
of radius $\delta:= (14\max(|k|, 1)+6)\varepsilon$ under the metric $\bar{d}_{n_i}$, since $(t, \Gamma g)$ can be chosen arbitrarily. It follows that
\[
s_{n_i}(\mathbb{T}\times\Gamma\backslash G, T_1, d, v, \delta)\leq\varepsilon^{-1}L^4 q_i^7.
\]
Since $\mathcal{Q}^\sharp(B)$ is infinite, we can let $q_i$ tend to infinity along $\mathcal{Q}^\sharp(B), $ getting
\begin{equation*}
\begin{split}
  \liminf_{n\rightarrow\infty}\frac {s_{n}(\mathbb{T}\times\Gamma\backslash G, T_1, d, v, \delta)}{n^\tau} &\leq\liminf_{\substack{i\rightarrow\infty \\ q_i\in\mathcal{Q}^\sharp(B),\ i\geq I_0}}\frac {s_{n_i}(\mathbb{T}\times\Gamma\backslash G, T_1, d, v, \delta)}{n_i^\tau} \\
    &\leq\liminf_{\substack{i\rightarrow\infty \\ q_i\in\mathcal{Q}^\sharp(B),\ i\geq I_0}}\frac {\varepsilon^{-1}L^4 q_i^7}{q_i^{8+\tau}}\leq\liminf_{\substack{i\rightarrow\infty \\ q_i\in\mathcal{Q}^\sharp(B),\ i\geq I_0}}\frac {\varepsilon^{-1}L^4}{q_i}=0.
\end{split}
\end{equation*}
Since $\varepsilon$ can be arbitrary small, this means that the measure complexity of $(\mathbb{T}\times\Gamma\backslash G, T, \rho)$ is weaker that $n^\tau$ when $\mathcal{Q}^\sharp(B)$ is infinite.

Finally, we deal with the case that $\mathcal{Q}^\sharp(B)$ is finite. Now $M_1(B)$ is also finite. Hence the conclusion of Lemma \ref{HLW4.2} still holds if we replace $M_2(B)$ by $\mathbb{Z}\backslash\{0\}. $ Hence the functions $\tilde{g}_\phi(t), $ $\tilde{g}_\eta(t), $ and $\tilde{g}_\psi(t)$ defined by
\[
\tilde{g}_\phi(t)=\sum_{m\neq0}\frac{\hat{\phi}(m)e(mt)}{e(m\alpha)-1}, \quad
\tilde{g}_\eta(t)=\sum_{m\neq0}\frac{\hat{\eta}(m)e(mt)}{e(m\alpha)-1}, \quad
\tilde{g}_\psi(t)=\sum_{m\neq0}\frac{\hat{\psi}(m)e(mt)}{e(m\alpha)-1}
\]
are continuous and periodic with period $1$. Thus we can write
\begin{align} \label{22}
&\phi(t)=\tilde{g}_\phi(t+\alpha)-\tilde{g}_\phi(t), \nonumber\\
&\eta(t)=\hat{\eta}(0)+\tilde{g}_\eta(t+\alpha)-\tilde{g}_\eta(t), \nonumber\\
&\psi(t)=\hat{\psi}(0)+\tilde{g}_\psi(t+\alpha)-\tilde{g}_\psi(t).\nonumber
\end{align}
Notice that $\hat{\phi}(0)=0, $ and so there is no constant term in the first equation.
Similarly to \eqref{10},  we define $\tilde{R}:\mathbb{T}\times\Gamma\backslash G\rightarrow\mathbb{T}\times\Gamma\backslash G$ by
\[
\tilde{R}:(t, \Gamma g)\mapsto
\left(
t, \Gamma g\begin {pmatrix}
1 & k \tilde{g}_{\phi}(t) & \frac 12 k \tilde{g}_{\phi}^2(t)-\frac 12 k \tilde{g}_{\eta}(t)+\tilde{g}_{\psi}(t)\\
0 & 1 & \tilde{g}_{\phi}(t)\\
0 & 0 & 1
\end{pmatrix}
\right).
\]
Then $\tilde{T}_1:=\tilde{R}^{-1}\circ T\circ \tilde{R}$ is given by
\[
\tilde{T}_1:(t, \Gamma g)\mapsto\begin {pmatrix}
t+\alpha, \Gamma g\begin {pmatrix}
1 & 0 & -\frac 12 k\hat{\eta}(0)+\hat{\psi}(0) \\
0 & 1 & 0\\
0 & 0 & 1
\end{pmatrix}
\end{pmatrix}
\]
as in $\eqref{12}$. Again by Lemma \ref{HLW5.2},
the measure complexity of $(\mathbb{T}\times\Gamma\backslash G, T, \rho)$ is weaker than $n^\tau$ if and only if the measure complexity of $(\mathbb{T}\times\Gamma\backslash G, \tilde{T}_1, v)$ is weaker than $n^\tau, $ where $v=\rho\circ \tilde{R} $.
However, $d$ is invariant under $\tilde{T}_1. $ So we have for any $n\geq1$ and $\varepsilon>0$ that
$$
s_{n}(\mathbb{T}\times\Gamma\backslash G, \tilde{T}_1, d, v, \varepsilon)=s_{1}(\mathbb{T}\times\Gamma\backslash G, \tilde{T}_1, d, v, \varepsilon).
$$
Since $\mathbb{T}\times\Gamma\backslash G$ is compact, we have $s_{1}(\mathbb{T}\times\Gamma\backslash G, \tilde{T}_1, d, v, \varepsilon)<\infty$ and consequently
$$
\lim_{n\rightarrow\infty}\frac {s_{n}(\mathbb{T}\times\Gamma\backslash G, \tilde{T}_1, d, v, \varepsilon)}{n^\tau}=\lim_{n\rightarrow\infty}\frac {s_{1}(\mathbb{T}\times\Gamma\backslash G, \tilde{T}_1, d, v, \varepsilon)}{n^\tau}=0.
$$
Hence the measure complexity of $(\mathbb{T}\times\Gamma\backslash G, T, \rho)$ is also weaker than $n^\tau$ if $\mathcal{Q}^\sharp(B)$ is finite. The proof is complete.
\end{proof}

\noindent\emph{Proof of Theorem \ref{thm1}}.
\ Note that the the skew product $T$ in Theorem \ref{thm1} is a special case of the skew product $S$ in Theorem \ref{thmrational}.
Theorem \ref{thm1} follows from Theorem \ref{thmrational}, 
Proposition \ref{prop:4} and Lemma \ref{HWY}.  \qed

\vskip 5mm
\noindent\textbf{Acknowledgments}\ 
This work is supported  by the National Natural Science Foundation of China (Grant No. 11771252).

\end{document}